\documentclass[12pt,reqno]{amsart}
\usepackage[left=2.5cm, right=2.5cm, top=2.5cm, bottom=2.5cm]{geometry}
\usepackage{amssymb}
\usepackage{mathtools}
\usepackage{thmtools}
\usepackage{thm-restate}

\usepackage[shortlabels]{enumitem}

\usepackage{mathrsfs}
\usepackage{bbm}
\usepackage{bm}

\usepackage{url}

\usepackage{nag}

\usepackage{ulem}

\usepackage[colorlinks=true,allcolors=blue]{hyperref}

\renewcommand{\subset}{\subseteq}

\newtheorem{theorem}            {Theorem}[section]
\newtheorem{corollary}          [theorem]{Corollary}
\newtheorem{proposition}        [theorem]{Proposition}
\newtheorem{lemma}              [theorem]{Lemma}
\theoremstyle{definition}
\newtheorem{definition}         [theorem]{Definition}
\newtheorem{remark}         [theorem]{Remark}
\newtheorem{example}            [theorem]{Example}

\newcommand{\set}[1]{{ \left\{ #1 \right\} }}

\renewcommand{\Re}[1]{{\operatorname{Re}#1}}
\renewcommand{\Im}[1]{{\operatorname{Im}#1}}
\newcommand{\ora}[1]{{\overrightarrow{#1}}}
\newcommand{\Dom}[1]{{\operatorname{Dom}{#1}}}

\def\cI{ {\mathcal I} }

\def\cN{ {\mathcal N} }

\def\cU{\mathcal U}

\newcommand{\sF}{\mathscr{F}}
\newcommand{\sG}{\mathscr{G}}
\newcommand{\sU}{\mathscr{U}}

\newcommand{\ep}{{\varepsilon}}

\newcommand{\vf}{{\varphi}}

\def\C{{\mathbb C}}

\def\RR{{\mathbb R}}
\def\ZZ{{\mathbb Z}}

\def\bbS{{\mathbb S}}

\newcommand{\bbx}{\mathbbm{x}}

\newcommand{\bbh}{\mathbbm{h}}

\newcommand{\x}{\bm{x}}

\newcommand{\h}{\bm{h}}

\newcommand{\wt}[1]{\widetilde{#1}}

\newcommand{\fralg}[1]{\langle #1 \rangle}
\makeatletter
\def\moverlay{\mathpalette\mov@rlay}
\def\mov@rlay#1#2{\leavevmode\vtop{
                \baselineskip\z@skip \lineskiplimit-\maxdimen
                \ialign{\hfil$#1##$\hfil\cr#2\crcr}}}
\makeatother

\newcommand{\plangle}{\moverlay{(\cr<}}
\newcommand{\prangle}{\moverlay{)\cr>}}
\newcommand{\skf}[1]{\plangle #1 \prangle}

\makeatletter
\DeclareFontFamily{OMX}{MnSymbolE}{}
\DeclareSymbolFont{MnLargeSymbols}{OMX}{MnSymbolE}{m}{n}
\SetSymbolFont{MnLargeSymbols}{bold}{OMX}{MnSymbolE}{b}{n}
\DeclareFontShape{OMX}{MnSymbolE}{m}{n}{
    <-6>  MnSymbolE5
   <6-7>  MnSymbolE6
   <7-8>  MnSymbolE7
   <8-9>  MnSymbolE8
   <9-10> MnSymbolE9
  <10-12> MnSymbolE10
  <12->   MnSymbolE12
}{}
\DeclareFontShape{OMX}{MnSymbolE}{b}{n}{
    <-6>  MnSymbolE-Bold5
   <6-7>  MnSymbolE-Bold6
   <7-8>  MnSymbolE-Bold7
   <8-9>  MnSymbolE-Bold8
   <9-10> MnSymbolE-Bold9
  <10-12> MnSymbolE-Bold10
  <12->   MnSymbolE-Bold12
}{}

\let\llangle\@undefined
\let\rrangle\@undefined
\DeclareMathDelimiter{\llangle}{\mathopen}%
                     {MnLargeSymbols}{'164}{MnLargeSymbols}{'164}
\DeclareMathDelimiter{\rrangle}{\mathclose}%
                     {MnLargeSymbols}{'171}{MnLargeSymbols}{'171}
\makeatother

\newcommand{\fpsx}{\llangle \bbx \rrangle}

\def\curl{\operatorname{curl}}

\newcommand{\GL}{\mathrm{GL}}

\mathchardef\mhyphen="2D

\def\mccarthy{M\raise.5ex\hbox{c}Carthy}

\newcommand{\bsbm}{\left[ \begin{smallmatrix}}
\newcommand{\esbm}{\end{smallmatrix} \right]}

\newcommand{\bbm}{ \begin{bmatrix}}
\newcommand{\ebm}{\end{bmatrix} }

\newcommand{\bpm}{\begin{pmatrix}}
\newcommand{\epm}{\end{pmatrix}}

\newcommand{\bspm}{\left(\begin{smallmatrix}}
\newcommand{\espm}{\end{smallmatrix}\right)}

\newcommand{\bsm}{\begin{smallmatrix}}
\newcommand{\esm}{\end{smallmatrix}}

\newcommand{\bal}{\begin{align*}}
\newcommand{\eal}{\end{align*}}

\newcommand{\df}[1]{{\bf{#1}}{\index{#1}}}
\newcommand{\dft}[1]{{\textit{#1}}{\index{#1}}}

\title[]{Free Potential Functions}

\author{Meric L. Augat}

\numberwithin{equation}{section}

\begin{document}

\begin{abstract}
This article establishes free versions of two classical theorems: derivatives are curl-free and every curl-free vector field (on a 
simply connected domain) is a derivative.
We show that the derivative of a noncommutative free analytic map must be free-curl free -- an analog of having zero curl. Moreover, under 
the assumption that the free domain is connected, this necessary condition is sufficient. Specifically, if $T$ is analytic free vector field 
defined on a connected free domain then $DT(X,H)[K,0] = DT(X,K)[H,0]$ if and only if there exists an analytic free map $f$ such that 
$Df(X)[H] = T(X,H)$.
\end{abstract}

\maketitle

\thispagestyle{empty}

\section{Introduction}

The free derivative is a pervasive tool in free analysis that provides deep insight into many interesting problems; recent papers discussing 
the free derivative provide stronger results than their classical counterparts \cite{Pas14,Aug18}.
In particular, the free inverse function theorem has been studied in a variety of settings \cite{Pas14,AK-V15,AM16,Man20}.

This paper answers the question {\it when is a free map the derivative of a free analytic function?}
While a resolution to the question can be found in \cite{K-VV14}, or more recently, \cite{K-VSV20}, the techniques used therein will be less 
familiar with a non-specialist of free analysis.
This article, on the other hand, is set up to be an analogue of two very well-known theorems in analysis: derivatives are curl-free and a 
curl-free vector field (on a simply connected domain) is a derivative.
The statements of our main theorems, their proofs, and even the definitions used within are natural generalizations of their classical 
counterparts.
By following a well-known proof arc, this article provides a satisfying proof that is accessible to both veterans and neophytes of free 
analysis.

\subsection{Preliminaries}

Throughout the paper we fix $g,h\in \ZZ^+$ and let $M(\C)^g = (M_n(\C)^g)_{n=1}^\infty.$
 It serves as our {\it universal} free set.
A \df{subset} $\Omega\subseteq M(\C)^g$ is sequence $\Omega = (\Omega[n])_{n=1}^\infty,$ where $\Omega[n]\subset M_n(\C)^g.$ 
 We  say $\Omega$ is a \df{free set} if  
\begin{enumerate}
	\item $X\oplus Y = (X_1\oplus Y_1, \dots, X_g\oplus Y_g) 
		= (\bspm X_{1\vphantom{g}} & 0 \\ 0 & Y_{1\vphantom{g}} \espm, \dots, \bspm X_g & 0 \\ 0 & Y_g \espm) \in \Omega[n+m]$;
	\item $S^{-1}XS = (S^{-1}X_1 S, \dots, S^{-1}X_g S)\in \Omega[n];$
\end{enumerate}
for all $m,n\in \ZZ^+$ and 
   $X = (X_1,\dots, X_g)\in \Omega[n]$, $Y = (Y_1,\dots, Y_g)\in \Omega[m]$ and $S\in \GL_n(\C)$.
If $\Omega$ is a free set then $\Omega \times M(\C)^g$ is defined to be the free set $(\Omega[n]\times M_n(\C)^g)_{n=1}^\infty$.

A free set $\Omega$ is said to be a \df{free domain} if each $\Omega[n]$ is open. 
We note that while our definition requires $\Omega$ to be closed under simultaneous conjugation by similarities, in certain settings it is 
desirable to assume the weaker condition of $\Omega$ being closed under simultaneous conjugation by unitaries, see \cite{JKMMP19,HKMV}.

If $f = (f[n])_{n=1}^\infty$ where $f[n]:\Omega[n]\to M_n(\C)^h$, then we write $f:\Omega\to M(\C)^h$.
If, in addition,
\begin{enumerate}
	\item $f(X\oplus Y) = \bpm f(X) & 0 \\ 0 & f(Y) \epm$
	\item $f(S^{-1}XS) = S^{-1}f(X)S$
\end{enumerate}
whenever $X = (X_1,\dots, X_g)\in \Omega[n]$, $Y = (Y_1,\dots, Y_g)\in \Omega[m]$ and $S\in \GL_n(\C)$ then $f$ is a \df{free map}.
A free map is \df{continuous} if each $f[n]$ is continuous and \df{analytic} if each $f[n]$ is analytic.
When $f$ is analytic, then its \df{nc directional derivative} at $X\in \Omega[n]$ in the direction of $H\in M_n(\C)^g$ is
\[
	Df(X)[H] = \lim_{z\to 0} \frac{f(X+zH) - f(X)}{z}.
\]

A basic result of free analysis \cite{HKM11,K-VV14} says a continuous free map on a free domain is analytic and 
\[
	f\bpm X & H \\  0  & X\epm = \bpm f(X) & Df(X)[H] \\ 0 & f(X) \epm.
\]

The nc directional derivative is a Fr{\' e}chet derivative, hence it is linear in the direction of the derivative.
By virtue of a few simple block matrix computations, the derivative map $Df:\Omega\times M(\C)^g\to M(\C)^h$ is easily seen to be free 
analytic and is linear in its second set of coordinates (such a map will be called a \dft{free vector field}). 
On the other hand, if we are given an analytic free vector field then we can ask whether it is the derivative of a free map.
This question is akin to asking when a given smooth vector field is the gradient of a smooth function.

Our two main results provide necessary and sufficient conditions for an analytic free vector field to be the derivative of an analytic free 
map. 
The first is a free analog of the Clairaut-Schwarz Theorem on the equality of mixed partial derivatives:

\begin{restatable*}[Free Clairaut-Schwarz Theorem]{theorem}{clairaut}
	\label{thm:free Clairaut}
	Suppose $\Omega$ is a free domain and $f:\Omega\to M(\C)^h$ is free analytic.
	If $F:\Omega\times M(\C)^g\to M(\C)^{h}$ is defined by $F(X,H) = Df(X)[H]$, then 
        $F$ is an analytic free vector field and {free curl-free}: 
	$DF(X,H)[K,L] = DF(X,K)[H,L]$ for all $X\in \Omega$ and $H,K,L\in M(\C)^g$.
\end{restatable*}

The second theorem guarantees the existence of a potential function when the analytic free vector field is free-curl free:

\begin{restatable*}{theorem}{potential}
	\label{thm:free-curl free means potential}
	Suppose $\Omega$ is a connected free domain and $T$ is an analytic free vector field on $\Omega\times M(\C)^g$.
	Then $T$ is free-curl free if and only if there exists a free analytic function $f$ on $\Omega$ such that $T(X,H) = Df(X)[H]$ for all 
	$X\in \Omega$ and $H\in M(\C)^g$.
\end{restatable*}

See \cite{Ste18,K-VSV20} for related results for difference-differential operators from a different perspective.

Finally, the proof of Theorem~\ref{thm:free-curl free means potential} requires several pieces, of which the following is of independent 
interest.

\begin{restatable*}{proposition}{similarity}
	\label{prop:analytic and unitary gives similarity}
	Suppose $\Gamma\subset M_n(\C)^g$ is nonempty, open and $\Gamma$ is similarity invariant.
	If $\beta:\Gamma\to M_n(\C)^h$ is analytic and respects conjugation by unitaries then $\beta$ respects conjugation by similarities.
\end{restatable*}

\subsection{Acknowledgments}

The author would like to thank Scott McCullough for his support throughout the writing of this paper, Robert Martin for kindly providing 
Theorem~\ref{thm:nc pluriharmonic conj} and its proof, and John {\mccarthy} for his valuable comments and suggestions as well as for asking 
the initial question that led to this paper.

\section{Necessity}

The first classical result investigated is the Clairaut-Schwarz Theorem on the equality of mixed partial derivatives: if $\phi$ is any scalar-valued $\mathcal{C}^2$ function, then $\curl(\nabla \phi) = 0$.
Our proof, below, of the free analog of this classical result relies only on our domain being closed under direct sums and \dft{unitary} conjugation, rather than direct sums and conjugation by similarities.

\begin{definition}
	Suppose $T:\Omega\times M(\C)^g\to M(\C)^h$. The map $T$ is said to be a \df{free vector field} if $T$ is free as a 
	function from $\Omega\times M(\C)^g$ to $M(\C)^h$ and $T$ is linear in its second coordinates: $T(X,cH+ K) = cT(X,H) + 
	T(X,K)$ for all $X\in \Omega$, $H,K\in M(\C)^g$ and $c\in \C$.

	Suppose that in addition to being a free vector field, $T$ is also analytic. 
	Then the \df{free-curl} of $T$ is the difference $DT(X,H)[K,0] - DT(X,K)[H,0]$ (for each $X$ we get a map $(H,K)\mapsto 
	DT(X,H)[K,0] - DT(X,K)[H,0]$).
	If the free-curl is zero for each $X\in \Omega$, then $T$ is said to be \df{free-curl free}.
\end{definition}

\begin{example}
	Let $T:\Omega\times M(\C)^g\to M(\C)^h$ with $T(X,H) = XH - HX$ and note that $T$ is a free vector field.
	In order for $T$ to be the derivative of a free map, we need $XH - HX = 0$ for all $X\in \Omega$ and $H\in M(\C)^g$.
	Hence our domain must only contain points $X$ such that $XH = HX$ for all $H$, that is $X$ is always a scalar matrix. However, the set of scalar matrices is not open and on this set, $XH - HX$ is identically zero.
\end{example}

\begin{lemma}
	\label{lem:curl with or without L}
	Suppose $\Omega$ is a free domain. 
	If $T$ is an analytic free vector field on $\Omega\times M(\C)^g,$ then $DT(X,H)[K,0] = DT(X,K)[H,0]$ if and 
	only if $DT(X,H)[K,L] = DT(X,K)[H,L]$ for all $L\in M(\C)^g$.
\end{lemma}

\begin{proof}
	Suppose $DT(X,H)[K,0] = DT(X,K)[H,0]$ and let $L$ be given. Since $T$ is a free vector field we immediately observe that 
	$DT(X,H)[0,L] = T(X,L)$.
	Thus,
	\begin{align*}
		DT(X,H)[K,L] 
			&= DT(X,H)[K,0] + DT(X,H)[0,L] = DT(X,K)[H,0] + T(X,L)\\
			&= DT(X,K)[H,0] + DT(X,K)[0,L] = DT(X,K)[H,L].
	\end{align*}
	The other direction is obtained immediately by choosing $L = 0$.
\end{proof}

\clairaut

\begin{proof}
	It is sufficient to show $DF(X,H)[K,0] = DF(X,K)[H,0]$ by Lemma~\ref{lem:curl with or without L}.
	Recall that for any free analytic map we have
	\[
		f\bpm X & H \\ 0 & X \epm = \bpm f(X) & F(X,H) \\ 0 & f(X) \epm.
	\]
	Since $F$ is free analytic,
	\[
		F\bpm (X,H) & (K,0) \\ (0,0) & (X,H) \epm = \bpm F(X,H) & DF(X,H)[K,0] \\ 0 & F(X,H) \epm.
	\]
	In particular,
	\begin{align*}
		\bpm F(X,H) & DF(X,H)[K,0] \\ 0 & F(X,H) \epm 
			&= F\bpm (X,H) & (K,0) \\ (0,0) & (X,H) \epm = F\left(\bpm X & K \\ 0 & X \epm, \bpm H & 0 \\ 0 & H\epm\right) \\
			&= Df\bpm X & K \\ 0 & X \epm \bbm H & 0 \\ 0 & H \ebm.
	\end{align*}
	Thus,
	\begin{align*}
		f\bpm 
			X & K & H & 0 \\ 
			0 & X & 0 & H \\
			0 & 0 & X & K \\
			0 & 0 & 0 & X
		\epm
		&= \bpm 
			\begin{matrix}
				f(X) & F(X,K) \\ 0 & f(X) \\ \vspace{-1em} \\ 0 & 0 \\ 0 & 0
			\end{matrix}
			& \begin{matrix} 
				Df\bpm X & K \\ 0 & X \epm\bbm H & 0 \\ 0 & H \ebm
			\\ \vspace{-1em} \\
			\begin{matrix}
				f(X) & F(X,K) \\ 0 & f(X)
			\end{matrix}
			\end{matrix}
		\epm \\
		&= \bpm 
				f(X) & F(X,K) & F(X,H) & DF(X,H)[K,0] \\ 0 & f(X) & 0 & F(X,H) \\ 0 & 0 & f(X) & F(X,K) \\ 0 & 0 & 0 & f(X)
			\epm.
	\end{align*}
	Moreover, letting
	\[
		U = 
		\bpm 
			1 & 0 & 0 & 0 \\ 
			0 & 0 & 1 & 0 \\ 
			0 & 1 & 0 & 0 \\ 
			0 & 0 & 0 & 1 
		\epm
	\]
	we see that
	\[
		\bpm 
			X & K & H & 0 \\ 
			0 & X & 0 & H \\
			0 & 0 & X & K \\
			0 & 0 & 0 & X
		\epm
		=
		U
		\bpm 
			X & H & K & 0 \\ 
			0 & X & 0 & K \\
			0 & 0 & X & H \\
			0 & 0 & 0 & X
		\epm	
		U^{-1}
	\]
	and the free structure of $f$ implies
	\[
		f\bpm 
			X & K & H & 0 \\ 
			0 & X & 0 & H \\
			0 & 0 & X & K \\
			0 & 0 & 0 & X
		\epm
		=
		U f\bpm 
			X & H & K & 0 \\ 
			0 & X & 0 & K \\
			0 & 0 & X & H \\
			0 & 0 & 0 & X
		\epm
		U^{-1}.
	\]
	Since
	\[
		f\bpm 
			X & K & H & 0 \\ 
			0 & X & 0 & H \\
			0 & 0 & X & K \\
			0 & 0 & 0 & X
		\epm
		=
		\bpm 
			f(X) & F(X,K) & F(X,H) & DF(X,H)[K,0] \\ 0 & f(X) & 0 & F(X,H) \\ 0 & 0 & f(X) & F(X,K) \\ 0 & 0 & 0 & f(X)
		\epm		
	\]
	and
	\begin{align*}
		U f\bpm 
			X & H & K & 0 \\ 
			0 & X & 0 & K \\
			0 & 0 & X & H \\
			0 & 0 & 0 & X
		\epm
		U^{-1}
		=
		U&\bpm 
			f(X) & F(X,H) & F(X,K) & DF(X,K)[H,0] \\ 0 & f(X) & 0 & F(X,K) \\ 0 & 0 & f(X) & F(X,H) \\ 0 & 0 & 0 & f(X)
		\epm U^{-1}\\
		=&\bpm 
			f(X) & F(X,K) & F(X,H) & DF(X,K)[H,0] \\ 0 & f(X) & 0 & F(X,H) \\ 0 & 0 & f(X) & F(X,K) \\ 0 & 0 & 0 & f(X)
		\epm			
	\end{align*}
	we conclude $DF(X,H)[K,0] = DF(X,K)[H,0]$.
	Therefore, for all $X\in \Omega$ and $H,K,L\in M(\C)^g$ we have $DF(X,H)[K,L] = DF(X,K)[H,L]$, as desired.
\end{proof}

\begin{corollary}
	Suppose $T$ is a free vector field on $\Omega\times M(\C)^g$.
	If there exist $X\in \Omega$ and $H,K\in M(\C)^g$ such that $DT(X,H)[K,0]\neq DT(X,K)[H,0]$, then $T$ is not the derivative of an 
	analytic free map on $\Omega$.
\end{corollary}

\begin{proof}
	This is the contrapositive of Theorem~\ref{thm:free Clairaut}.
\end{proof}

Theorem~\ref{thm:free Clairaut} gives us a clean necessary condition for an analytic free vector field to be the derivative of a free 
analytic map.

\begin{example}
	Once more, let us set $T(X,H) = XH - HX$. Since
	\[
		DT(X,H)[K,0] = KH - HK \neq HK - KH = DT(X,K)[H,0]
	\]
	on any open free set, we see that $T$ is not the derivative of an analytic free map.	
\end{example}

\begin{example}
	Similarly, in the classical setting $(-y,x)$ is a standard example of a non-conservative vector field. 
	In our case choosing the free vector field $T(X_1,X_2,H_1,H_2) = X_1H_2 - H_1X_2$, we see that
	\[
		DT(X,H)[K,0] = K_1H_2 - H_1K_2 \neq H_1K_2 - K_1H_2 = DT(X,K)[H,0]
	\]
	unless $K_1H_2 = H_1K_2$ for all $K,H\in M(\C)^2$ -- a preposterous idea!
	Thus, $T$ is not the derivative of an analytic free map.
\end{example}

\begin{remark}
	As mentioned above, the fact that the proof of Theorem~\ref{thm:free Clairaut} relies only on $\Omega$ being open and invariant under 
	direct sums and unitary conjugation implies that it can be extended to the operator NC setting.
	However, the sufficiency proofs rely on our domain being closed under conjugation by similarities as well as evaluations on finite operators.	
\end{remark}

\begin{remark}
	Many of the results found in this section ostensibly could be obtained by restricting a free function to a fixed level $n$ and 
	identifying it 	as an $n^2h$ tuple of analytic maps in $n^2g$ commuting variables.
	The classical theory then applies and gives us necessary and sufficient condition for the existence of a scalar potential function. 
	However, this misses the free forest for the classical trees.
\end{remark}

\section{Sufficiency}
In the classical setting, if $\curl(\Phi) = 0$ on a simply connected domain, then there exists a scalar potential function for $\Phi$.
In outline, the proof of this fact proceeds by showing 
 that curl free implies path independent (on a simply connected domain)
 to guarantee that  a potential function constructed 
via line integrals from an {\it anchor} point in the domain is well defined
 (independent of the choice of path).

In the free setting  free-curl free implies path independent regardless of the geometry
 of the domain $\Omega.$ Proving that the natural candidate
 for a potential function is a free function requires some care.

\begin{definition}
	Suppose $\Omega$ is a free domain and suppose $T$ is an analytic free vector field on $\Omega\times M(\C)^g$.
	For any smooth path $\gamma:[0,1]\to \Omega[n]$ the entries of $T(\gamma(t),\gamma'(t))$ are analytic functions of $t$, 
	and we define 
	\[
		\mathcal{I}(T,\gamma) := \int_0^1 T(\gamma(t),\gamma'(t)) \, dt
	\]
	to be the result of applying the integral entry-wise.{\footnote{
	Any smooth path will be bounded away from the boundary of $\Omega[n]$, hence the integrals defined above 
	will have no convergence issues.}}
	
	We say $T$ is \df{path independent} if for all $n\in \ZZ^+$, whenever $\gamma_1,\gamma_2:[0,1]\to 
	\Omega[n]\times M_n(\C)^g$ are smooth, $\gamma_1(0) = \gamma_2(0)$ and $\gamma_1(1) = \gamma_2(1)$, then
	\[
		\mathcal{I}(T,\gamma_1) = 
                   \int_0^1 T(\gamma_1(t), \gamma_1'(t)) \, dt = \int_0^1 T(\gamma_2(t), \gamma_2'(t)) \, dt = \cI(T,\gamma_2).
	\]
\end{definition}

\begin{lemma}
	\label{lem:int is free}
	Suppose $\Omega$ is a free domain and $T$ is an analytic free vector field. If $\gamma:[0,1]\to \Omega[n]$ and $\eta:[0,1]\to \Omega[m]$ 
	are 
	smooth paths and $S\in \GL_n(\C)$ then
	\[
		\cI(T,\gamma \oplus \eta) = \cI(T,\gamma) \oplus \cI(T,\eta).
	\]
	and
	\[
		\cI(T,S^{-1}\gamma S) = S^{-1}\cI(T,\gamma)S.
	\]
\end{lemma}

\begin{proof}
	Note $(S^{-1}\gamma S)'(t) = S^{-1}\gamma'(t)S$ and $(\gamma\oplus \eta)'(t) = \gamma'(t)\oplus \eta'(t)$.
	Hence, the linearity of integration and the free nature of $T$ yield
	\[
		\cI(T,\gamma\oplus \eta) 
			= \int_0^1 T(\gamma\oplus \eta, \gamma'\oplus \eta') \, dt
			= \int_0^1 T(\gamma, \gamma') \oplus T(\eta, \eta')\, dt
			= \cI(T,\gamma)\oplus \cI(T,\eta)
	\]
	and
	\[
		\cI(T,S^{-1}\gamma S) 
			= \int_0^1 T(S^{-1}\gamma S, S^{-1}\gamma' S) \, dt 
			= \int_0^1 S^{-1} T(\gamma, \gamma') S \, dt
			= S^{-1} \cI(T,\gamma) S.
	\]
\end{proof}

\begin{proposition}
	\label{prop:free-curl free means path independent}
	Suppose $\Omega$ is a free domain and $T$ is an analytic free vector field.
	If $T$ is free-curl free on $\Omega$ (that is, $DT(X,H)[K,0] = DT(X,K)[H,0]$), then $T$ is path independent.
\end{proposition}

\begin{proof}
	Our first observation is that since the derivative of a free function is a Fr{\'e}chet derivative, it respects the chain rule and the  
	Fundamental Theorem of Calculus:
	\begin{equation}
		\label{eq:FTC}
		\int_0^1 D\alpha(\gamma(t))[\gamma'(t)] \, dt = \int_0^1 (\alpha\circ \gamma)'(t) \, dt = \alpha(\gamma(1)) - \alpha(\gamma(0)).
	\end{equation}
	
	Fix $n\in \ZZ^+$ such that $\Omega[n]\neq \varnothing$ and suppose $\gamma_1$ and $\gamma_2$ are smooth paths in $\Omega[n]$ with the 
	same endpoints.
	Let
	\[
		S = 
		\bpm 
			1 & 0 & 0 & -1 \\ 
			0 & 1 & 0 & 0 \\ 
			0 & 0 & 1 & 0 \\ 
			0 & 0 & 0 & 1
		\epm
		\quad \text{ and } \quad
		\hat{\gamma}(t) =
		\bpm 
			\gamma_1(t) & 0 & 0 & 0 \\ 
			0 & \gamma_2(t) & 0 & 0 \\ 
			0 & 0 & \gamma_1(t) & 0 \\ 
			0 & 0 & 0 & \gamma_2(t)
		\epm
	\]
	and observe that
	\[
		S \hat{\gamma}(t) S^{-1}
		=
		\bpm 
			\gamma_1(t) & 0 & 0 & \gamma_1(t)-\gamma_2(t) \\ 
			0 & \gamma_2(t) & 0 & 0 \\ 
			0 & 0 & \gamma_1(t) & 0 \\ 
			0 & 0 & 0 & \gamma_2(t)
		\epm.
	\]
	In particular, $S\hat{\gamma}(t)S^{-1}$ when viewed as a $2\times 2$ block matrix is precisely in the form $\bspm X & H \\ 0 & X \espm$; 
	precisely the form needed for realizing the derivative via point evaluation.
	Hence with a choice of $\gamma_{1,1'} = (\gamma_1, \gamma_1')$ and $\gamma_{2,2'} = (\gamma_2, \gamma_2')$ and an application of 
	Lemma~\ref{lem:int is free} we see
	\begin{equation}
		\label{eq:DT gamma}
	\begin{split}
		T(S\hat{\gamma}(t)S^{-1}, \, &S\hat{\gamma}'(t)S^{-1}) = \\
		S&\bpm 
			T(\gamma_{1,1'}) & 0 & 0 & 0 \\ 
			0 & T(\gamma_{2,2'}) & 0 & 0 \\ 
			0 & 0 & T(\gamma_{1,1'}) & 0 \\ 
			0 & 0 & 0 & T(\gamma_{2,2'})
		\epm S^{-1} \\
		=
		&\bpm 
			T(\gamma_{1,1'}) & 0 & 0 & T(\gamma_{1,1'}) - T(\gamma_{2,2'}) \\ 
			0 & T(\gamma_{2,2'}) & 0 & 0 \\ 
			0 & 0 & T(\gamma_{1,1'}) & 0 \\ 
			0 & 0 & 0 & T(\gamma_{2,2'})
		\epm \\
		=
		&\bpm 
			\begin{matrix}
				T(\gamma_{1,1'}) & 0 \\ 0 & T(\gamma_{2,2'}) \\ \vspace{-1em} \\ 0 & 0 \\ 0 & 0
			\end{matrix}
			& \begin{matrix} 
				DT\bpm \gamma_{1,1'} & 0 \\ 0 & \gamma_{2,2'} \epm\bbm 0 & \gamma_{1,1'}-\gamma_{2,2'} \\ 0 & 0 
				\ebm
			\\ \vspace{-1em} \\
			\begin{matrix}
				T(\gamma_{1,1'}) & 0 \\ 0 & T(\gamma_{2,2'})
			\end{matrix}
			\end{matrix}
		\epm.
	\end{split}
	\end{equation}
	Expanding out the derivative in the upper right hand corner and applying the fact that we assumed $T$ is free-curl free:
	\begin{equation}
		\label{eq:DT chain rule}
	\begin{split}
		DT\bpm \gamma_{1,1'} & 0 \\ 0 & \gamma_{2,2'} \epm
		&\bbm 0 & \gamma_{1,1'}-\gamma_{2,2'} \\ 0 & 0 \ebm	\\
		&=
		DT\left(
			\bpm \gamma_1 & 0 \\ 0 & \gamma_2 \epm,
			\bpm \gamma_1' & 0 \\ 0 & \gamma_2' \epm
		\right)\left[
			\bpm 0 & \gamma_1 - \gamma_2 \\ 0 & 0 \epm,
			\bpm 0 & \gamma_1'-\gamma_2' \\ 0 & 0 \epm
		\right]	\\
		&= 
		DT\left(
			\bpm \gamma_1 & 0 \\ 0 & \gamma_2 \epm,
			\bpm 0 & \gamma_1 - \gamma_2 \\ 0 & 0 \epm
		\right)\left[
			\bpm \gamma_1' & 0 \\ 0 & \gamma_2' \epm,
			\bpm 0 & \gamma_1'-\gamma_2' \\ 0 & 0 \epm
		\right].
	\end{split}
	\end{equation}
	This is precisely the derivative of the composition of $T$ with a smooth path.
	Choosing
	\[
		\tilde{\gamma}(t) 
		= \left(			
			\bpm \gamma_1 & 0 \\ 0 & \gamma_2 \epm,
			\bpm 0 & \gamma_1 - \gamma_2 \\ 0 & 0 \epm
		\right)
	\]
	and applying equations~\eqref{eq:FTC} and \eqref{eq:DT chain rule} we see
	\[
		\cI(DT,\tilde{\gamma}) = \int_0^1 DT( \tilde{\gamma}(t))[\tilde{\gamma}'(t)] \, dt = T(\tilde{\gamma}(1)) - T(\tilde{\gamma}(0)).
	\]
	However,
	\[
		\tilde{\gamma}(1)
		= \left(			
			\bpm \gamma_1(1) & 0 \\ 0 & \gamma_2(1) \epm,
			\bpm 0 & \gamma_1(1) - \gamma_2(1) \\ 0 & 0 \epm
		\right)
		= \left(			
			\bpm \gamma_1(1) & 0 \\ 0 & \gamma_2(1) \epm,
			\bpm 0 & 0 \\ 0 & 0 \epm
		\right)
	\]
	and similarly
	\[
		\tilde{\gamma}(0)
		= \left(			
			\bpm \gamma_1(0) & 0 \\ 0 & \gamma_2(0) \epm,
			\bpm 0 & \gamma_1(0) - \gamma_2(0) \\ 0 & 0 \epm
		\right)
		= \left(			
			\bpm \gamma_1(0) & 0 \\ 0 & \gamma_2(0) \epm,
			\bpm 0 & 0 \\ 0 & 0 \epm
		\right).
	\]
	Thus, linearity of $T$ in its second coordinate implies $T(\tilde{\gamma}(1)) = 0 = T(\tilde{\gamma}(0))$, hence $\cI(DT,\tilde{\gamma}) 
	= 0$.
	In light of equation~\eqref{eq:DT gamma} it follows that
	\[
		\cI(T,S \hat{\gamma} S^{-1}) =
		\bpm 
			\cI(T,\gamma_1) & 0 & 0 & 0 \\ 
			0 & \cI(T,\gamma_2) & 0 & 0 \\ 
			0 & 0 & \cI(T,\gamma_1) & 0 \\ 
			0 & 0 & 0 & \cI(T,\gamma_2)
		\epm.
	\]
	Finally, another application of Lemma~\ref{lem:int is free} yields
	\[
		\cI(T,S \hat{\gamma}S^{-1}) = S\cI(T,\hat{\gamma})S^{-1} =
		\bpm 
			\cI(T,\gamma_1) & 0 & 0 & \cI(T,\gamma_1) - \cI(T,\gamma_2) \\ 
			0 & \cI(T,\gamma_2) & 0 & 0 \\ 
			0 & 0 & \cI(T,\gamma_1) & 0 \\ 
			0 & 0 & 0 & \cI(T,\gamma_2)
		\epm .
	\]
	Therefore
	\[
		\cI(T,\gamma_1) = \int_0^1 T(\gamma_1(t), \gamma_1'(t)) \, dt = \int_0^1 T(\gamma_2(t), \gamma_2'(t)) \, dt = \cI(T,\gamma_2)
	\]
	and we conclude $T$ is path independent.
\end{proof}

\begin{remark}
Peculiarly enough, there was no mention of the geometry of $\Omega$ as one expects in the classical setting.
This is a common phenomenon in Free Analysis and for some insight into this curiosity, we use an argument from \cite{Pas20}.
If $\gamma$ and $\eta$ are two distinct non-intersecting paths in $\Omega$ with the same endpoints then setting
\[
	\Gamma = \set{
		\bspm a & -b \\ b & a \espm \bspm \gamma(t) & 0 \\ 0 & \eta(t) \espm\bspm a & b \\ -b & a \espm 
		\, : \, a,b\in \RR, \, a^2+b^2 = 1, \, t\in [0,1]		
	}\subset \Omega
\]
we see that $\Gamma\cong S^2$, which is simply connected. 
\end{remark}

We know from Theorem~\ref{thm:free Clairaut} that derivatives of free maps are free-curl free.
If $\Omega$ is connected, then Theorem~\ref{thm:free-curl free means potential} provides the converse.
However, for the first time in this article, the proof of Theorem~\ref{thm:free-curl free means potential} requires a delicate touch.
Specifically, by defining our potential level-wise, we open up the possibility that the different levels do not align.
%

In the classical setting, if $F$ is a path-independent vector field, then fix an ``anchor" point $x_0$ in the domain and define a potential 
function $\vf$ with $\vf(x)$ equal to the result of integrating $F$ from $x_0$ to $x$ along any path.
We still employ this exact idea in the free setting: for each level $n$, we choose an ``anchor" point $Z_n$ in $\Omega[n]$ and define 
$\alpha_n(X)$ to be the integral of $T$ from $Z_n$ to $X$, along any path.
However, there is no guarantee that the $Z_n$ are related in any obvious fashion (beware: there exist free domains that 
contain no scalar matrices), nor do we have any guarantee that $\alpha_n$ respects similarities (or even unitaries).
Thus, the proof of Theorem~\ref{thm:free-curl free means potential} is an imitation of the classical proof with three additional steps to fix 
any misalignments in the level-wise potentials:
\begin{enumerate}[(i)]
	\item At each level $n$, fix an ``anchor" point $Z\in \Omega[n]$ and define $\alpha_n(X) = \int_{Z_n}^X T$, integrated along any path.
	\item Define $\beta_n(X)$ as the Haar integral of $U^*\alpha_n(UXU^*)U$ over the unitary group -- this results in a function that 
	respects conjugation by unitaries.
	\item Use the entry-wise analyticity of $\beta_n$ to show that $\beta_n$ respect similarities
	\item Use level-wise direct sums to find constants $b_n$ such that $\Phi_n = \beta_n + b_n$ defines an analytic free map.
\end{enumerate}

Step (iii) is interesting in its own right and may have application outside of this paper.
Thus, we present Step (iii) as a self-contained Proposition.

\similarity

\begin{proof}
	Let $\mathbb{S}_n$ denote the set of $n\times n$ self-adjoint matrices.
	We claim that if $\eta:M_n(\C)\to \C^k$ is analytic and vanishes on $\mathbb{S}_n$, then $\eta$ vanishes on $M_n(\C)$.
	
	Accordingly, suppose $\eta$ is an analytic map vanishing on $\mathbb{S}_n$.
	Let $A,B\in \mathbb{S}_n$ and define the map $\zeta:\C\to \C^k$ by $\zeta(z) = \eta(A+izB)$.
	Note that if $w$ is pure imaginary then $A+iwB\in \mathbb{S}_n$ and we must have $\zeta(w) = 0$.
	Hence, $\zeta$ is an analytic map that vanishes on the imaginary axis, thus $\zeta$ is identically zero.
	In particular, $0 = \zeta(1) = \zeta(A+iB)$.
	Every $X\in M_n(\C)$ can be decomposed as $A+iB$ for some $A,B\in \mathbb{S}_n$, therefore $\eta$ vanishes on $M_n(\C)$ and our claim is proved.

	Let $\cU_n$ denote the group of $n\times n$ unitary matrices.
	Define $\psi:\GL_n(\C)\to M_n(\C)^h$ by 
	\[
		\psi(S) = S\beta(S^{-1}XS)S^{-1} - \beta(X).
	\]
	We see immediately that $\psi$ is analytic and vanishes on $\cU_n$.
	Next, let $\ep:M_n(\C)\to GL_n(\C)$ be defined by $\ep(X) = e^{iX}$. 
	Note that $\ep$ is surjective and $\ep$ maps $\mathbb{S}_n$ onto $\cU_n$.
	Hence, the composition $\psi\circ \ep:M_n(\C)\to M_n(\C)^h$ is analytic and vanishes on $\mathbb{S}_n$.
	By our claim, $\psi\circ \ep = 0$, hence the surjectivity of $\ep$ implies $\psi$ vanishes on $\GL_n(\C)$.
	Therefore $\beta$ respects conjugation by similarities.
\end{proof}

\potential

\begin{proof}
	The first direction is handled by Theorem~\ref{thm:free Clairaut}.

	Conversely, suppose $T$ is an analytic free vector field on $\Omega\times M(\C)^g$.
	Let $\cN$ be the set of all positive integers $n$ such that $\Omega[n]\neq \varnothing$.
	We first construct an analytic free map $\Phi$ on $\Omega\times \Omega$ by
	\[
		\Phi(X,Y) = \int_{0}^{1} T(\gamma(t),\gamma'(t))\, dt
	\]
	where $\gamma$ is any smooth path in $\Omega$ such that $\gamma(0) = Y$ and $\gamma(1) = X$.
	This map is well-defined since Proposition~\ref{prop:free-curl free means path independent} tells us that $T$ is path independent.
	Moreover, $\Phi(X,Y) + \Phi(Y,Z) = \Phi(X,Z)$ for all $X,Y,Z\in \Omega[n]$ and all $n\in \cN$.
	
	Let $\gamma$ and $\eta$ be smooth paths from $Y$ to $X$ and $Z$ to $W$, respectively.
	Hence $S^{-1}\gamma S$ is a path from $S^{-1}YS$ to $S^{-1}XS$ while $\gamma \oplus \eta$ is a path from $Y\oplus Z$ to $X\oplus W$.
	Thus, Lemma~\ref{lem:int is free} shows us that $\Phi$ is free.
	Moreover, for any smooth path $\gamma$ from $Y$ to $X$ and any smooth path $\eta$ from $0$ to $H$,
	\begin{align*}
		\Phi\left(\bpm X & H \\ 0 & X\epm, \bpm Y & 0 \\ 0 & Y\epm\right) 
			&= \int_0^1 T\left( \bpm \gamma(t) & \eta(t) \\ 0 & \gamma(t) \epm, \bpm \gamma'(t) & \eta'(t) \\ 0 & \gamma'(t) \epm\right) \, dt\\
			&= \int_0^1 \bpm T(\gamma(t),\gamma'(t)) & DT(\gamma(t),\gamma'(t))[\eta(t),\eta'(t)] \\ 0 & T(\gamma(t),\gamma'(t)) \epm \, dt\\
			&= \int_0^1 \bpm T(\gamma(t),\gamma'(t)) & DT(\gamma(t),\eta(t))[\gamma'(t),\eta'(t)] \\ 0 & T(\gamma(t),\gamma'(t)) \epm \, dt\\
			&= \bpm \Phi(X,Y) & T(\gamma(1),\eta(1)) - T(\gamma(0),\eta(0)) \\ 0 & \Phi(X,Y) \epm \\
			&= \bpm \Phi(X,Y) & T(X,H)-T(Y,0) \\ 0 & \Phi(X,Y) \epm\\
			&= \bpm \Phi(X,Y) & T(X,H) \\ 0 & \Phi(X,Y) \epm.
	\end{align*}
	Thus, 
	\begin{equation}
		\label{eq:DF is T}
		D\Phi(X,Y)[H,0] = T(X,H) \quad \text{ and } \quad D\Phi(X,Y)[0,K] = -T(Y,K).
	\end{equation}
	
	Next, for each $n\in \cN$, choose $Z_n\in \Omega[n]$ and define $\alpha_n:\Omega[n]\to M_n(\C)^h$ by 
	\[
		\alpha_n(X) = \Phi(X,Z_n).
	\]
	Using the fact that $\Phi(X,Y) = -\Phi(Y,X)$ we see
	\[
		\alpha_n(X) - \alpha_n(Y) = \Phi(X,Z_n) - \Phi(Y,Z_n) = \Phi(X,Y)
	\]
	for all $X,Y\in \Omega[n]$.

	For each $n\in \cN$ let $\cU_n$ denote the group of $n\times n$ unitary matrices and define $\beta_n:\Omega[n]\to M(\C)^h$ via Haar integration:
	\[
		\beta_n(X) = \int_{\cU_n} U^*\alpha_n(UXU^*)U\, dU.
	\]
	Let $V\in \cU_n$ and note that
	\begin{align*}
		\beta_n(V^*XV) 
			&= \int_{\cU_n} U^*\alpha_n(UV^*XVU^*)U \, dU \\ 
			&= \int_{\cU_n} V^*W^*\alpha_n(W^*XW^*)WV \, d(WV) \\ 
			&= V^*\Big[\int_{\cU_n} W^*\alpha_n(W^*XW^*)W \, dW\Big]V \\ 
			&= V^*\beta_n(X)V,
	\end{align*}
	where the invariance of the Haar measure is used.
	Moreover, for any $X,Y\in \Omega[n]$ we have
	\begin{align*}
		\beta_n(X) - \beta_n(Y) 
			&= \int_{\cU_n} U^*\big(\alpha_n(U^*XU) - \alpha_n(U^*YU)\big)U \, dU \\
			&= \int_{\cU_n} U^*\big( \Phi(U^*XU,U^*YU) \big)U \, dU \\
			&= \int_{\cU_n} \Phi(X,Y) \, dU \\
			&= \Phi(X,Y).
	\end{align*}
	Hence, Equation~\ref{eq:DF is T} implies $D\beta_n(X)[H] = D\Phi(X,Y)[H,0] = T(X,H)$.
	Since $\beta_n$ is analytic and respects conjugation by unitaries, Proposition~\ref{prop:analytic and unitary gives similarity} implies 
	$\beta_n$ also respects conjugation by similarities.
	
	Our last step is to construct $f$ from $\beta$ by adding appropriate scalars. 
	Suppose $X\in\Omega[m]$, $Y\in \Omega[n]$ and let
	\[
		\beta_{m+n}\bpm X & 0 \\ 0 & Y \epm = \bpm A & B \\ C & D \epm,
	\]
	where $A,B,C,D$ are tuples of size $m\times m, m\times n, n\times m$ and $n\times n$, respectively.
	Take any nonzero $\mu,\nu\in\C$ and note $S = \mu I_m\oplus \nu I_n$ is invertible and $S^{-1}(X\oplus Y)S = X\oplus Y$.
	Hence 
	\[
		\bpm A & B \\ C & D \epm = \beta_{m+n}\big(S^{-1}(X\oplus Y)S\big) 
		= S^{-1}\beta_{m+n}(X\oplus Y)S
		= \bpm A & \frac{\nu}{\mu} B \\ \frac{\mu}{\nu}C & D \epm
	\]
	and we conclude that $B$ and $C$ are zero.
	
	Next, for any $X_1,X_2\in \Omega[m]$ and $Y_1,Y_2\in\Omega[n]$ we see that
	\begin{equation}
		\label{eq:A D indep of X Y}
		\begin{split}
		\beta_{m+n}\bpm X_1 & 0 \\ 0 & Y_1 \epm - \beta_{m+n}\bpm X_2 & 0 \\ 0 & Y_2 \epm
			&= \Phi\left(\bpm X_1&0\\0&Y_1 \epm, \bpm X_2&0\\0&Y_2 \epm \right) \\
			&= \bpm \Phi(X_1,X_2)&0\\0&\Phi(Y_1,Y_2) \epm.
		\end{split}
	\end{equation}
	Hence, $A$ is independent of our choice of $Y$ and by a similar argument, $D$ is independent of our choice of $X$.
	Treating $A$ and $D$ as functions we see from Equation~\eqref{eq:A D indep of X Y} that
	\[
		A(X_1) - A(X_2) = \Phi(X_1,X_2) = \beta_m(X_1) - \beta_m(X_2)
	\]
	and
	\[
		D(Y_1)-D(Y_2) = \Phi(Y_1,Y_2) = \beta_n(Y_1) - \beta_n(Y_2).
	\]
	Rearranging these equations shows
	\[
		A(X_1) - \beta_m(X_1) = A(X_2) - \beta_m(X_2)
	\]
	for all $X_1,X_2\in \Omega[m]$. 
	Hence $A - \beta_m$ is constant as must be $D - \beta_n$.
	We let $A - \beta_m = C_m$ and $D - \beta_n = C_n$.

	Now take $S_m$ and $S_n$ to be similarities of size $m\times m$ and $n\times n$, respectively.
	Observe
	\begin{align*}
			\bpm A(S_m^{-1}XS_m) & 0 \\ 0 & D(S_n^{-1}YS_n) \epm
			&= \beta_{m+n}\bpm S_m^{-1}XS_m & 0 \\ 0 & S_n^{-1}YS_n \epm \\
			&= \bpm S_m^{-1} & 0 \\ 0 & S_m^{-1} \epm 
				\beta_{m+n}\bpm X & 0 \\ 0 & Y \epm \bpm S_m & 0 \\ 0 & S_n \epm \\
			&= \bpm S_m^{-1}A(X)S_m & 0 \\ 0 & S_n^{-1}D(Y)S_n \epm.
	\end{align*}
	Hence, $A$ and $D$ respect similarities.
	Moreover, 
	\[
		C_m = A(S_m^{-1}XS_m) - \beta_m(S_m^{-1}XS_m) 
			= S_m^{-1}\Big(A(X) - \beta_m(X)\Big)S_m
			= S_m^{-1}C_mS_m.
	\]
	Thus, $C_m = c_mI_m$ for some scalar tuple $c_m\in \C^h$.
	A similar argument shows that $C_n = c_nI_n$ for some $c_n\in \C^h$.
	Therefore, $\beta_{m+n}$ ``nearly" respects direct sums.
	
	For any $m,n\in \cN$ we let $c_{m+n}^m$ and $c_{m+n}^n$ be the scalars in $\C^h$ such that
	\[
		\beta_{m+n}\bpm X & 0 \\ 0 & Y \epm - \bpm \beta_m(X) & 0 \\ 0 & \beta_n(Y) \epm
			= \bpm c_{m+n}^m I_m & 0 \\ 0 & c_{m+n}^n I_n \epm
	\]
	for any $X\in \Omega_m$ and $Y\in \Omega_n$.
	We note that these constants are well-defined since $\beta_{m+n}$ respects similarities:
	\begin{align*}
		\beta_{m+n}\bpm Y & 0 \\ 0 & X \epm &- \bpm \beta_n(Y) & 0 \\ 0 & \beta_m(X) \epm\\
			&= 	\bpm 0 & I_n \\ I_m & 0 \epm
				\left[\beta_{m+n}\bpm X & 0 \\ 0 & Y \epm - \bpm \beta_m(X) & 0 \\ 0 & \beta_n(Y) \epm\right]
				\bpm 0 & I_m \\ I_n & 0 \epm \\
			&= 	\bpm 0 & I_n \\ I_m & 0 \epm \bpm c_{m+n}^m I_m & 0 \\ 0 & c_{m+n}^nI_n \epm 
				\bpm 0 & I_m \\ I_n & 0 \epm 
			= 	\bpm c_{m+n}^n I_n & 0 \\ 0 & c_{m+n}^mI_m  \epm.
	\end{align*}
	
	Suppose now that $k,m,n\in \cN$ and $X\in \Omega[k]$, $Y\in \Omega[m]$ and $Z\in \Omega[n]$. 
	It follows that
	\begin{align*}
		\beta_{k+m+n}\bspm X & 0 & 0 \\ 0 & Y & 0 \\ 0 & 0 & Z \espm 
			- \bspm \beta_k(X) & 0 & 0 \\ 0 & \beta_m(Y) & 0 \\ 0 & 0 & \beta_n(Z) \espm
			&= \beta_{k+m+n}\bspm X & 0 & 0 \\ 0 & Y & 0 \\ 0 & 0 & Z \espm 
			- \bspm \beta_{k+m}\bspm X & 0 \\ 0 & Y \espm & 0 \\ 0 & \beta_n(Z)  \espm\\
			&+ \bspm \beta_{k+m}\bspm X & 0 \\ 0 & Y \espm & 0 \\ 0 & \beta_n(Z)  \espm
			- \bspm \beta_k(X) & 0 & 0 \\ 0 & \beta_m(Y) & 0 \\ 0 & 0 & \beta_n(Z) \espm\\
			&= \bspm c_{k+m+n}^{k+m}I_{k+m} & 0 \\ 0 & c_{k+m+n}^nI_n \espm
			+ \bspm c_{k+m}^k I_k & 0 & 0\\ 0 & c_{k+m}^m I_m & 0 \\0 & 0 & 0 \espm\\
			&= \bspm (c_{k+m+n}^{k+m} + c_{k+m}^k )I_k & 0 & 0\\ 0 & (c_{k+m+n}^{k+m} + c_{k+m}^m) I_m & 0 \\0 & 0 & c_{k+m+n}^n I_n \espm .
	\end{align*}
	Interchanging the roles of $X, Y$ and $Z$, we have the (redundant) equations
	\begin{equation}
		\label{eq:triple cs}
		\begin{split}
		c_{k+m+n}^k &= c_{k+m+n}^{k+m}+c_{k+m}^k = c_{k+m+n}^{k+n} + c_{k+n}^k \\
		c_{k+m+n}^m &= c_{k+m+n}^{k+m}+c_{k+m}^m = c_{k+m+n}^{m+n} + c_{m+n}^m \\
		c_{k+m+n}^n &= c_{k+m+n}^{k+n}+c_{k+n}^n = c_{k+m+n}^{m+n} + c_{m+n}^n.
		\end{split}
	\end{equation}
	
	Now, let $n_0 = \min(\cN)$ and for each $k\in \cN$, let $b_k = c_{k+n_0}^k - c_{k+n_0}^{n_0}$ and define
	\[
		f_k(X) = \beta_k(X) + b_k I_k
	\]
	for all $X\in \Omega[k]$.
	Since $f_k$ differs from $\beta_k$ by a scalar matrix, $f_k$ respects conjugation by similarities.
	Setting $f = (f_k)_{k\in \cN}$, we claim $f:\Omega\to M(\C)^h$ also respects direct sums.
	Accordingly, suppose $k,m\in \cN$, $X\in \Omega[k]$ and $Y\in \Omega[m]$ and consider
	\[
		f_{k+m}\bpm X & 0 \\ 0 & Y \epm - \bpm f_k(X) & 0 \\ 0 & f_m(Y) \epm.
	\]
	We show that this difference is zero.
	By their respective definitions, 
	\begin{align*}
		&f_{k+m}\bpm X & 0 \\ 0 & Y \epm - \bpm f_k(X) & 0 \\ 0 & f_m(Y) \epm \\
			&= \beta_{k+m}\bpm X & 0 \\ 0 & Y \epm - \bpm \beta_k(X) & 0 \\ 0 & \beta_m(Y) \epm
			+ b_{k+m}I_{k+m} - \bpm d_k I_k & 0 \\ 0 & b_m I_m \epm \\
			&= \bpm c_{k+m}^k I_k & 0 \\ 0 & c_{k+m}^mI_m\epm + b_{k+m}I_{k+m} - \bpm b_k I_k & 0 \\ 0 & b_m I_m \epm.
	\end{align*}
	The first $k\times k$ block is $(c_{k+m}^k + b_{k+m} - b_k)I_k$.
	By our definition of $b_{k+m}$ and $b_k$ and a few applications of Equation~\eqref{eq:triple cs} we see
	\begin{align*}
		c_{k+m}^k + b_{k+m} - b_k 
			&= c_{k+m}^k + [c_{k+m+n_0}^{k+m} - c_{k+m+n_0}^{n_0}] - [c_{k+n_0}^k - c_{k+n_0}^{n_0}] \\
			&= c_{k+m}^k + [(c_{k+m+n_0}^{k}-c_{k+m}^k) - c_{k+m+n_0}^{n_0}] - [c_{k+n_0}^k - c_{k+n_0}^{n_0}] \\
			&= (c_{k+m+n_0}^{k} - c_{k+n_0}^k) - (c_{k+m+n_0}^{n_0} - c_{k+n_0}^{n_0}) \\
			&= c_{k+m+n_0}^{k+n_0} - (c_{k+m+n_0}^{k+n_0}) \\
			&= 0.
	\end{align*}
	Hence, for all $k,m\in \cN$, $f_{k+m}(X\oplus Y) = f_k(X)\oplus f_m(Y)$ for all $X\in \Omega[k]$ and $Y\in \Omega[m]$.
	Thus, $f$ respects direct sums and conjugation by similarities.
	Moreover, $Df(X)[H] = D\beta(X)[H] = T(X,H)$ since at each level, $f$ and $\beta$ differ by a scalar matrix.
	Therefore there exists an analytic free function $f$ such that $Df(X)[H] = T(X,H)$.
\end{proof}

\begin{corollary}
	Suppose $\Omega$ is a free domain that is connected.
	If $(\phi_n)_{n=1}^\infty$ is a sequence of level-wise analytic functions on $\Omega$ such that $(D\phi_n)_{n=1}^\infty$ is an analytic 
	free vector field, then there exists an analytic free map $f$ on $\Omega$ such that $\phi$ and $f$ differ by a constant at each level if 
	and only if $(D\phi_n)_{n=1}^\infty$ is free-curl free.
\end{corollary}

\begin{proof}
	Suppose $T = (D\phi_n)_{n=1}^\infty$. 
	If $T$ is free-curl free then Theorem~\ref{thm:free-curl free means potential} tells us there exists a free analytic map $f$ on $\Omega$ such that $Df(X)[H] = T(X,H)$.
	At level $n$, $Df(X)[H] = T(X,H) = D\phi_n(X)[H]$, hence $f$ and $\phi$ must differ by a constant.
	
	On the other hand, if $f$ and $\phi$ differ by a constant at each level, then their derivatives are the same, hence $Df(X)[H] = T(X,H) = D\phi(X)[H]$.
	Since $f$ is free analytic, Theorem~\ref{thm:free Clairaut} implies $T$ is free-curl free.
\end{proof}

\section{Applications}

In this section we present two applications of the ideas of the main article. 
The first result proves the existence of free pluriharmonic conjugates while the second shows that if a derivative is an nc rational, then it 
must be the derivative of an nc rational.

\subsection{Conjugates of Free Pluriharmonic Functions}
In this subsection we present an elementary proof of the existence of conjugates of free pluriharmonic functions.
While this fact has been previously established, see \cite[Corollary 2.2]{Pas20}, we provide an alternate proof taking advantage of 
Theorem~\ref{thm:free-curl free means potential}.

The ideas and proofs in this subsection are due to Robert Martin who kindly gave permission for the author to reproduce his work.

\begin{definition}
	Recall that $\bbS_n$ denotes the set of $n\times n$ self-adjoint matrices.
	Let $\bbS^g = (\bbS^g_n)_{n=1}^\infty$ denote our NC self-adjoint universe.
	For each $n$, let $\cU_n$ denote set of unitary matrices in $M_n(\C)$.
	A subset $\Gamma\subset \bbS^d$ is a \df{real free set} if it closed under direct sums and conjugation by joint unitary similarities. 
	
	Suppose $\Gamma\subset \bbS^d$ is a real free set.
	If $u = (u[n])_{n=1}^\infty$ where $u[n]:\Gamma[n]\to \bbS_n$, then we write $u:\Gamma\to \bbS$.
	If, in addition,
	\begin{enumerate}
		\item $u(X\oplus Y) = \bpm u(X) & 0 \\ 0 & u(Y) \epm$
		\item $u(V^{-1}XV) = V^{-1}u(X)V$
	\end{enumerate}
	whenever $X = (X_1,\dots, X_g)\in \Gamma[n]$, $Y = (Y_1,\dots, Y_g)\in \Gamma[m]$ and $V\in \cU_n$, then $u$ is a \df{real free map} or 
	\df{real free function}.	
\end{definition}

If $\Gamma$ is a real free set that is closed under conjugation by similarities, then Proposition~\ref{prop:analytic and unitary gives 
similarity} implies that any real free map on $\Gamma$ is an honest-to-goodness free map.

\begin{definition}
	If $X = \Re{X} + i\Im{X}\in M_n(\C)^d$, we write $\ora{X} := (\Re{X}, \Im{X})\in \bbS^{2g}$, and $\wt{X} := (-\Im{X}, \Re{X}) = 
	\ora{iX}$.
	We say $\Gamma\subset \bbS^g$ is a \df{real free domain} if $\Gamma = \ora{\Omega}$ for some free domain $\Omega$ and if $\Omega$ is 
	connected, then we say $\Gamma$ is \df{connected}.
	
	Let $f$ be an analytic free map on a free domain $\Omega\subset M(\C)^g$.
	If $u = \Re{f}$, then we view $u$ as a real free function defined on the real free domain
	\[
		\Dom{u} := \set{\ora{X} \, : \, X \in \Omega}
	\]
\end{definition}

The following results are found in \cite[Thoerem 4.1]{HorKle20} and \cite[Theorem 3.5.2]{Klem20}, respectively.

\begin{theorem}[NC Cauchy-Riemann equations] 
	\label{thm:NC Cauchy-Riemann equations}
	Suppose $f$ is an analytic free map on a free domain $\Omega$. If $u = \Re{f}$ and $v = \Im{f}$, then
	\[
		Du(\ora{X})[\ora{H}] = Dv(\ora{X})[\wt{H}].
	\]
\end{theorem}

\begin{theorem}[NC Laplace equations] 
	\label{thm:NC Laplace equations}
	Suppose $f$ is an analytic free map on a free domain $\Omega$. If $u = \Re{f}$ and $v = \Im{f}$, then
	\[
		D^2u(\ora{X})[\ora{H}, \ora{K}] + D^2u(\ora{X})[\wt{H}, \wt{K}] = 0
	\]
	and
	\[
		D^2v(\ora{X})[\ora{H}, \ora{K}] + D^2v(\ora{X})[\wt{H}, \wt{K}] = 0.
	\]
\end{theorem}

\begin{definition}(\cite[Definition 3.5.6]{Klem20}).
	A real free function, $u$, on a real free domain, $\Gamma$, is \df{free pluriharmonic} if it obeys the NC Laplace equations.
\end{definition}

We need one last technical result before proceeding.

\begin{theorem} \textup{(\cite[Theorems 3.5.2 \& 3.5.3, Corollary 3.5.4]{Ham82}).}
	\label{thm:C2 function result}
	Suppose $\sF$ and $\sG$ are Fr{\'e}chet spaces, $\sU\subset \sF$ is an open domain, and $T:\cU\to \sG$ is a continuous map.
	If $D^2T$ is jointly continuous as a map on $\sU\times \sF\times \sF$, then $D^2T(X)[Y,Z]$ is linear in both $Y$ and $Z$ and 
	$D^2T(X)[Y,Z] = D^2T(X)[Z,Y]$.
\end{theorem}

We now present the main theorem of this subsection and its proof.

\begin{theorem}
	\label{thm:nc pluriharmonic conj}
	Suppose $u:\Gamma\to\bbS$ is a free pluriharmonic function on the connected real free domain $\Gamma\subset \bbS^{2g}$, and assume that 
	$u$ is jointly level-wise continuous on $\Gamma\times \bbS^{2g}\times \bbS^{2g}$.
	Then $u$ has a free pluriharmonic conjugate $v:\Gamma\to \bbS$ so that $f = u + iv:\Omega\to M(\C)$ is an analytic free map on a 
	connected free domain $\Omega$ with $\ora{\Omega} = \Gamma$.
\end{theorem}

\begin{proof}
	Define $\Omega = \set{Z = X + iY \, : \, \ora{Z} = (X,Y)\in \Gamma}$.
	By the assumptions on $\Gamma$, $\Omega$ is a connected free domain. For each $n\in \ZZ^+$, $X\in \Omega[n]$ and $H\in M_n(\C)^g$, we 
	define
	\[
		T(X,H) := Du(\ora{X})[\ora{H}] - i Du(\ora{X})[\wt{H}].
	\]
	Note that if $u$ were the real part of an analytic free map, $f$, then by the NC Cauchy-Riemann equations, $Df(X)[H]$ must have this form.
	Next, we observe that $T$ is graded, preserves direct sums and conjugation by unitaries.
	Since $T$ is linear in $H$ it follows that it is a free vector field, moreover Proposition~\ref{prop:analytic and unitary gives 
	similarity} implies $T$ is an analytic free vector field.
	Thus, Theorem~\ref{thm:free-curl free means potential} implies that $T(X,H) = Df(X)[H]$ for some analytic free map $f$ if and only if $T$ 
	is free-curl free.
	
	With the free-curl of $T$ in mind, we note
	\[
		DT(X,H)[K,0] = D^2u(\ora{X})[\ora{H},\ora{K}] - i D^2u(\ora{X})[\wt{H},\ora{K}]
	\]
	and
	\[
		DT(X,K)[H,0] = D^2u(\ora{X})[\ora{K},\ora{H}] - i D^2u(\ora{X})[\wt{K},\ora{H}].
	\]
	Thus, if 
	\begin{equation}
		\label{eq:free curl part 1}
		D^2u(\ora{X})[\ora{H},\ora{K}] = D^2u(\ora{X})[\ora{K},\ora{H}]
	\end{equation}
	and
	\begin{equation}
		\label{eq:free curl part 2}
		D^2u(\ora{X})[\wt{H},\ora{K}] = D^2u(\ora{X})[\wt{K},\ora{H}]
	\end{equation}
	are both true, then it follows that $T$ is free-curl free.
	
	Our hypothesis on $u$ and Hamilton's Theorem~\ref{thm:C2 function result} immediately imply Equation~\eqref{eq:free curl part 1}.
	Next we observe
	\begin{align*}
		D^2u(\ora{X})[\wt{H}, \ora{K}] 
			&= D^2u(\ora{X})[\ora{K},\wt{H}] \text{ by Equation~\eqref{eq:free curl part 1}}\\
			&= -D^2u(\ora{X})[\wt{K},\wt{\wt{H}}] \text{ by NC Laplace equations}\\
			&= D^2u(\ora{X})[\wt{K},\ora{H}].
	\end{align*}
	Hence, $T$ is free-curl free and Theorem~\ref{thm:free-curl free means potential} implies
	\[
		Du(\ora{X})[\ora{H}] - iDu(\ora{X})[\wt{H}] = T(X,H) = Df(X)[H]
	\]
	for an analytic free map $f$ with free domain $\Omega$.
	By virtue of our construction,
	\[
		Du(\ora{X})[\ora{H}] = \Re{Df(X)[H]}
	\]
	so that we can choose $f$ with $u = \Re{f}$.
	Therefore $v = \Im{f}$ is a free pluriharmonic conjugate of $u$.
\end{proof}

\subsection{NC rational derivatives are derivatives of nc rationals}
The titular proposition of this subsection is proved using rational formal power series, so we set about defining the appropriate objects.
Let $\bbx = \set{x_1,\dots, x_g}$ and $\bbh = \set{h_1,\dots, h_g}$ be sets of freely noncommuting indeterminates.
The set $\fralg{\bbx}$ is the free monoid, i.e. the set of all words formed from the letters $x_1,\dots, x_g$.
The free algebra $\C\fralg{\bbx}$ is the set of all finite $\C$-linear combinations of elements of $\fralg{\bbx}$.

A \textit{formal power series} $S$ is a function $S:\fralg{\bbx}\to \C$ and the image of $w$ under $S$ is the \textit{coefficient} of $w$, 
denoted $[S,w]$.
The set of all formal power series (in $\bbx$ over $\C$) is denoted $\C\fpsx$.
For any word $w\in \fralg{\bbx}$ and series $S\in \C\fpsx$, we let $w^{-1}S = \sum_{v\in \fralg{\bbx}} [S,wv]v$ (effectively the backwards 
shift by $w$).

A subset $M$ of $\C\fpsx$ is called \textit{stable} if, for all $S\in M$ and $w\in \fralg{\bbx}$, the series $w^{-1}S\in M$.
A series $S$ is \df{rational} if and only if it is contained in a stable finitely generated left sub-module of $\C\fpsx$.
The set of all rational formal power series in $\C\fpsx$ forms an algebra and is denoted by $\C\skf{\bbx}_0$ (this notation is due to the 
fact that these series can exactly by realized as noncommutative rational functions with $0$ in their domain).
If $S$ is rational, then $w^{-1}S$ is rational for any $w\in \bbx$.

\begin{proposition}
	\label{prop:nc rational deriv is deriv of nc rational}
	Suppose $f\in \C\fpsx$.
	If $Df(\x)[\h]$ is a rational formal power series then $f(\x)$ is as well.
\end{proposition}

\begin{proof}
	For each $1\leq i\leq g$ we let $\partial_i f := Df(\x)[0,\dots,0,h_i,0,\dots,0]$.
	The map $\zeta_i(F(\x)[\h]) = F(\x)[0,\dots, h_i, \dots, 0]$ is an algebra homomorphism and since $Df(\x)[\h]$ is a rational series by 
	hypothesis, it follows that $\partial_i f = \zeta_i(Df(\x)[\h])$ is a rational series as well.
	Consequently, $h_i^{-1}\partial_i f$ is rational.
	Next, we write
	\[
		f = c_0 + \sum_{j=1}^g x_j f_j
	\]
	for some $f_i\in \C\fpsx$.
	Hence,
	\[
		\partial_i f = h_if_i + \sum_{j=1}^g x_j \partial_i f_j\in \C\skf{\bbx,\bbh}_0.
	\]
	As noted above, the rationality of $\partial_i f$ implies that $h_i^{-1}\partial_i f$ is rational as well.
	Since $h_i^{-1}\partial_if_i = f_i$, we have that each $f_i$ is rational.
	Therefore, $f = c_0 + \sum_{i=1}^g x_i f_i$ is a rational formal power series.
\end{proof}

\begin{remark}
	Not every nc rational can be represented as a rational formal power series, however, every nc rational can be represented as a rational 
	generalized series over $M_n(\C)\fpsx$, for some $n$, see \cite{Vol18}. Moreover, the proof of Proposition~\ref{prop:nc rational deriv is 
	deriv of nc rational} works in exactly the same way for generalized series, allowing us to conclude that every nc rational derivative is 
	indeed the derivative of an nc rational.
\end{remark}

\bibliographystyle{alpha}
\bibliography{FreePotentials}

\end{document}